\documentclass{amsart}
\usepackage{amsfonts}
\usepackage{amsmath}
\usepackage{enumerate}

\newtheorem{theorem}{Theorem}

\newtheorem{corollary}[theorem]{Corollary}

\newtheorem{definition}[theorem]{Definition}

\newtheorem{lemma}[theorem]{Lemma}

\newtheorem{remark}[theorem]{Remark}

\newcommand{\E}{\mathbf{E}}
\newcommand{\Pp}{\mathbf{P}}
\newcommand{\R}{\mathbb{R}}

\newcommand{\Z}{\mathbb{Z}}

\newcommand{\eps}{\varepsilon}

\newcommand{\Fc}{\mathcal{F}}

\title{Quadratic covariation estimates in nonsmooth stochastic calculus}
\author{ Sergio Angel Almada Monter and Yuri Bakhtin }

\begin{document}

\maketitle

\begin{abstract}
Given a Brownian Motion $W$, in this paper we study the asymptotic behavior, as $\eps \to 0$, of the quadratic covariation between $f ( \eps W)$ and $W$ in the case in which $f$ is not smooth. Among the main features discovered is that the speed of the decay in the case $f \in C^\alpha$ is at least polynomial in $\eps$ and not exponential as expected. We use a recent representation as a backward- forward It\^o integral of $[ f ( \eps W), W]$  to prove an $\eps$-dependent approximation scheme which is of independent interest. We get the result by providing estimates to this approximation. The results are then adapted and applied to generalize the results of~\cite{MioNonl}, and~\cite{nhn} related to the Small Noise Exit from a Domian problem for the Saddle Case.


\end{abstract}

\section{Introduction}
%Quadratic variation and covariation is a central topic in stochastic analysis. Its study started with the classical change 
%of variable formula by It\^o~\cite{Ito},  and it has gained increasing attention due to recent generalizations of this formula. 
%In this paper, we provide estimates for the quadratic covariation between a Brownian Motion with small intensity and a (nonsmooth) function of itself.
%The key feature of the classical 

One of the central results of stochastic calculus is It\^o's  change of variables formula for twice 
differentiable transformations of semimartingales. 
It was realized recently that one also needs to study nonlinear maps 
that are not smooth enough to allow an application of the classical It\^o formula. Various approaches to less regular changes of variables have
 been introduced, see~\cite{Boleau},~\cite{BoleauYor},~\cite{EisenbaumIto},~\cite{FollmerQuad},~\cite{ItoC1}, and
references therein.
These studies show that the key feature of the It\^o formula, the quadratic covariation term, is well-defined under 
much weaker assumptions than those leading to the traditional formula. However, no nontrivial quantitative estimates of the arising quadratic covariation
processes have appeared in the literature, to the best of our knowledge.

One area where such estimates are naturally needed is small random perturbations of dynamical systems. Often,
in the course of a study of a stochastic system one has to make a simplifying change of coordinates, transforming the system locally to a simpler one. If the transformation map
is $C^2$, then one can apply the classical It\^o calculus and easily control the It\^o correction term. However, there are situations where
a natural change of variables is less regular than $C^2$, and in these cases there is no readily available tool that could be used to control 
the generalized It\^o correction. 

The goal of this paper is to close this gap and provide quantitative estimates on the generalized It\^o
correction term under nonclassical assumptions on the transformation. 

Let us now be more precise. Let $W$ be a standard $1$-dimensional Wiener process on a complete probability space $(\Omega,\Sigma,\Pp)$ and $\eps>0$ be a constant.
If $g: \R \to \R$ is $C^2$, then the classical It\^o formula is (see~\cite[Section II.7]{ProtterLibro})
\[
 g(\eps W(t))-g(0)=\eps \int_0^t g'(\eps W(s))dW(s)+\frac{\eps^2}{2}\int_0^t g''(\eps W(s))ds.
\]
Introducing $f=g'\in C^1$,  we can also rewrite the second term in the r.h.s.\ as quadratic covariation between $f(\eps W)$ and $\eps W$:
for $Q_\eps(t)= [f(\eps W), \eps W](t)$, we have
\begin{equation*}  \label{eqn: quad-C1}
Q_\eps(t)=\eps^2 \int_0^t f' (\eps W(s) )ds,\quad t\ge 0.
\end{equation*}
In particular, for any $T>0$,  $\eps^{-1} \sup_{ t \leq T} Q_\eps(t) \to 0$ in probability as $\eps \to 0$.  In this paper we show that this converges holds in the case in which $f$ is not differentiable.  

The motivation for this problem relies on small random perturbation of dynamical systems. Suppose that $b$ is a vector field with a critical point at $x^*$ and let $S$ denote the flow generated by $b$: \[
\frac{d}{dt}S^t x = b(S^t x), \quad S^0 x = x.
\]
It is well known (see Section 2.8 of \cite{Perko}) that there is a continuous change of variables $g$ so that locally around $g(x^*)$ the flow $g(S^tx)$ behaves like the linearized version of $S$. In the small random perturbation case, this combined with the traditional It\^o formula imply (see, e.g. \cite{nhn},~\cite{MioNonl}) that if $g$ is at least $C^2$, then the system 
\[
dX_\eps (t) = b(X_\eps (t) ) dt + \eps dW(t), \quad X_\eps (0) = x_0,
\]
could be analyzed by working with the linear system 
\[
d \tilde X_\eps (t) = \left( A \tilde X_\eps (t) + \frac{\eps^2}{2} \Phi_\eps (X^\eps (t) )  \right) dt + \eps \sigma (X_\eps (t) ) dW(t), \quad \tilde X_\eps (0) = g(x_0),
\]
where $x_0$ is close enough to $x^*$, $A$ is the Jacobian of $b$ at $x^*$, $\sigma$ is at least a continuous matrix valued function, and  $\eps^2 \Phi_\eps$ is the term corresponding to the quadratic covariation between $g^\prime (X_\eps)$ and $X_\eps$. There are well established cases for which $g$ is known to be $C^1$, see e.g. Hartman Theorem on Section 2.8 of~\cite{Perko}. In these cases, an already known $C^1$ formulation of It\^o's formula implies a similar analogy between the non-linear and linear systems. Hence, estimates that show that in these cases the quadratic covariation term decays faster than the It\^o term allow to reduce the local analysis to simpler exit problem for Ornstein-Uhlenbeck processes. 

 The analysis of the quadratic covariation $[g'(X),X]$ in connection
with extensions of It\^o's formula for functions $g\notin C^2$ is fundamental for nonsmooth 
It\^o calculus, see~\cite{EisenbaumLocal},~\cite{EisenbaumChapter},~\cite{FollmerQuad},~\cite{ItoC1},~\cite{ItoCov}.
In~\cite{FollmerQuad},~\cite{ItoC1},~\cite{ItoCov} methods from backward stochastic calculus were used (see also the summary~\cite{RussoCov}), 
while in ~\cite{EisenbaumLocal},~\cite{EisenbaumChapter} a local time approach was used.
 The basic result that has been explained
in the cited literature from several points of view is that for $T>0$, 
\begin{equation} \label{eqn: dif_intro}
  Q_\eps (t) = - \eps \int_0^t f(\eps W(s) )dW(s) - \eps \int_{T-t}^T f( \eps W(T-s) ) dW(T-s), 
\end{equation}
where both integrals can be understood as It\^o integrals w.r.t.\ appropriate filtrations. It is well known~\cite[page 389]{ProtterLibro} 
that the integral with respect to $W(T-\cdot)$ in~\eqref{eqn: dif_intro} is the time reversal 
of a semimartingale w.r.t.\ the natural filtration of $W(T-\cdot)$. Here the time reversal (with respect to $T>0$) 
of a process $X$ is understood as $X(T-t)-X(T)$. 

In this paper we exploit this structure by constructing an approximation scheme for $Q_\eps$
and using martingale techniques to show its consistency. As far as we know, this is the first attempt to use such a scheme in small noise analysis.
See~\cite{ValloisQuad} for a related but different scheme for local time approximation in the case $\eps=1$. 

The text is organized as follows. In Section~\ref{sec: Prelim} we state our main results that include the martingale representation for the quadratic covariation, and, in Section~\ref{sec: Saddle}, the results related to the application of non-smooth calculous to a particular small noise problem. In Section~\ref{sec: smal_quad} we use the martingale representation to propose an approximation scheme that we then use to 
prove the key bound that the main results depend upon. The proofs of the main theorems 
are given in Section~\ref{sec: proof_thm}. In Section~\ref{sec: add_proof} proofs of auxiliary lemmas are given.

\section{Main results} \label{sec: Prelim}
We are going to study $Q_\eps(t)= [f(\eps W), \eps W](t)$ assuming that $f:\R \to \R$  is a bounded and uniformly H\"older or Lipschitz function, although
these assumptions on $f$ can be relaxed.
It is convenient to formulate these assumptions in terms of modulus of continuity defined by:
\[
{ \rm osc}_f(\delta)=\sup_{ |t-s|< \delta} | f(s) - f(t) |,\quad \delta>0.
\]

Throughout the text, we work with an arbitrary fixed number $T>0$. We will not be explicit when including the
dependency on $T>0$ in the notation. We are ready to state the main results of the text.

\begin{theorem} \label{thm: Main} 
Suppose ${ \rm osc}_f(\delta) \le C_f \delta^\alpha$ for some $\alpha \in (0,1)$, $C_f >0$, and all sufficiently small $\delta$. Then,
 for every $\delta>0$, $\gamma \in (0,\alpha)$, and $\mu\in(\gamma,\alpha)$, there are constants $\eps_{\delta,\mu}>0$ and $C_{\delta,\mu}>0$  such that
\[
 \Pp \left\{ \eps^{-(1+\gamma)} \sup_{t \leq T} |Q_\eps (t)| > \delta  \right \} \leq C_{\delta,\mu} \eps^{ 2 (\alpha - \mu)/(1-\alpha) },\quad \eps \in (0, \eps_{\delta,\mu} ).\]
In particular, for any $\gamma\in (0,\alpha)$,
\[
 \eps^{-(1+\gamma)} \sup_{t \leq T} |Q_\eps (t)| \stackrel{\Pp}{\to} 0,\quad\eps\to 0. 
\]
\end{theorem}
This result is stronger than our initial claim that $\eps^{-1}Q_\eps\to 0$. Moreover, if $\alpha$ is close to $1$,  the exponent $1+\gamma$ can be chosen to be close to $2$. 

The method we employ to prove this theorem produces the following estimate in the Lipschitz case where $\alpha=1$: 

\begin{theorem} \label{thm: Main_Lip} 
Suppose ${ \rm osc}_f(\delta) \le C_f \delta$, for some constant $C_f >0$ and sufficiently small $\delta>0$. Then, for every $\delta>0$,
$\gamma \in (0,1)$, and $\mu\in(\gamma,1)$, there are constants $\eps_{\delta,\mu}>0$ and $C_{\delta,\mu}>0$  such that
\[
 \Pp \left\{ \eps^{-(1+\gamma)} \sup_{t \leq T} |Q_\eps (t)| > \delta  \right \} \leq C_{\delta,\mu}
e^{ - \eps^{ - (1 - \mu) } },\quad \eps\in (0,\eps_{\delta,\mu}).
\]
\end{theorem}

This theorem establishes that the rate of decay in probability is exponential in the Lipschitz case, which is coherent with the differentiable case in which almost sure convergence holds.  For the Holder case, the method only shows a polynomial upper bound which in principle does not imply that the convergence rate can not be exponential.

The proof of Theorems~\ref{thm: Main} and~\ref{thm: Main_Lip} will be given in Section~\ref{sec: proof_thm}. An important part of the analysis is
Theorem~\ref{thm: MainII} given in Section~\ref{sec: smal_quad} and in principle one can 
apply that result and its possible extensions to less regular functions $f$.  
The proof of Theorem~\ref{thm: MainII} is in turn
based on a forward-backward martingale representation of the quadratic covariation that we explain in Section~\ref{sec: FB}. For now, we proceed to explain an application of the above results to a small noise problem studied in~\cite{MioNonl},~\cite{Bakhtin-SPA},~\cite{nhn},~\cite{Day}, and~\cite{Kifer}.

\subsection{Applications to the Small Noise Problem.} 
\label{sec: Saddle}
In this section we consider the small noise scape from a saddle problem that, up to our knowledge, was first studied in~\cite{Kifer}. The objective of the section is to establish the role that Theorems~\ref{thm: Main}, and~\ref{thm: Main_Lip}  have in the study of this problem. Let us start with the statement of the problem. 

Consider a vector field $b:\R^d \to \R^d$, and a domain (open, bounded and convex set) $U \subset \R^d$ such that $0 \in U$ is the only critical point of $b$ in the closure of $U$. That is, $0 \in U$ is the only $x\in \bar{U}$ such that $b(x)=0$. Further, suppose that the vector field $b$ is such that its Jacobian at $0$, $A=Db(0)$ has at least one eigenvalue with positive real part, and one eigenvalue with negative real part. Under this conditions, consider the flow $S$ generated by $b$:
\[
\frac{d}{dt}S^tx = b( S^tx), \quad S^0x = x,
\]
and its small noise perturbation, 
\[
dX_\eps (t) = b( X_\eps )dt + \eps dW(t).
\]
The scape from a saddle problem is the study of the asymptotic behavior of the exit time 
\[
\tau_\eps(x) = \inf \left\{t>0: X_\eps(t) \in \partial U \right\}, \quad x \in U,
\]
and the exit location $X_\eps ( \tau_\eps (x) )$. The case of interest for this problem is when the initial condition for the diffusion $X_\eps(0)$ lies in the invariant stable manifold $$\mathcal{M}^s = \left \{  x: S^t x \to 0, \text{ as } t \to \infty \right \}.$$

The problem was first solved using a PDE approach in~\cite{Kifer}. In that paper, it is shown that the exit time is asymptotically logarithmic in $\eps$ and that the exit location is concentrated on the intersection of $\partial U$ and the invariant unstable manifold $$\partial U \cap \mathcal{M}^u = \partial U \cap \left \{  x: S^t x \to 0, \text{ as } t \to -\infty \right \}.$$
Later,~\cite{Day} refined the result of the exit distribution in two dimensions, and further refinements were made in higher dimensions in~\cite{Bakhtin-SPA}.

In~\cite{nhn} a further generalization to the exit location was obtained, using the idea mentioned in the introduction of this paper. This result was later iterated to get the first result for a heteroclinic network, which is a more general case than the simple saddle case. The argument in~\cite{nhn} is as follows. It is well known (see Section 2.8 of \cite{Perko}) that there is a continuous change of variables $h$ so that locally around $h(0)$ the flow $h(S^tx)$ behaves like the linearized version of $S$. For $X_\eps$ traditional It\^o formula imply (see, e.g. \cite{nhn},~\cite{MioNonl}) that if $h$ is at least $C^2$, then $\tilde X_\eps = h(X_\eps)$ satisfies
\[
d \tilde X_\eps (t) = \left( A \tilde X_\eps (t) + \frac{\eps^2}{2} \Phi_\eps ( \tilde X_\eps (t) )  \right) dt + \eps \sigma (X_\eps (t) ) dW(t), \quad \tilde X_\eps (0) = h(X_\eps(0) ),
\]
where $x_0$ is  to $0$, $\sigma$ is at least a continuous matrix valued function, and  $\eps^2 \Phi_\eps ( \tilde X_\eps (t) ) $ is the quadratic covariation term between the derivative of $h$ evaluated at $X_\eps$ and $X_\eps$ itself. Under the assumption that $h \in C^2$, the above converges to $0$ faster than the noise and hence it has no effect on the computation of the exit location. The limitation of this method is that the assumption $h \in C^2$ is quite restrictive. In~\cite{MioNonl} this restriction was studied by classifying systems that don't admit such a transformation $h$ in the $C^2$ class. The results of~\cite{nhn} were extended in~\cite{MioNonl} in the two dimensional setting: the change of coordinates $h$ transforms $X_\eps$ to a specific polynomial drift SDE in two dimensions which is then solved. In~\cite{MioNonl} it is also shown that this approach can not immediately be generalized to the high dimensional case.  As a consequence, in this paper, we attack the high dimensional case by following the approach proposed in~\cite{nhn} but by allowing the transformation $h$ to be not smooth. 

We focus on a particular case to keep the exposition manageable. The novelty relies in the assumption on the smoothness for the change of coordinates, which is the main focus of the paper. The proof is a rearrangement of the main facts covered in the body of the paper, and its presented in Section~\ref{sec: ProofThmSmallNoise}.  The theorem is stated in the spirit of Theorem 1 of~\cite{MioNonl}.

\begin{theorem} \label{thm: SmallNoise}
Suppose that $A$ has spectrum $\lambda_1,...,\lambda_d$ of real and simple eigenvalues such that $ \lambda_1 >  ...> \lambda_{\nu-1} > 0  > \lambda_\nu >... > \lambda_d$ for some integer $\nu \leq d$. Also, assume that $h:U \to \R^d$ is a differentiable function with differentiable inverse, such that all its partial derivatives satisfy the conditions of Theorem~\ref{thm: Main} with $\alpha > 1/2$ and that $h(S^tx) = e^{At}h(x)$ in $U= (-\Delta, \Delta)^d$. 

Denote $\partial U \cap \mathcal{M}^u = \{q_-, q_+\}$, and assume that $X_\eps(0) = x_0 \in \mathcal{M}^s \cap U$. Then, there is a family of random vectors $( \phi_\eps )_{\eps >0}$, a family of random variables $( \psi_\eps )_{ \eps > 0 }$, and a number  
\[
\beta = \left\{
	\begin{array}{ll}
		1,  & \mbox{if } \nu = 2 \text{ and } -\lambda_\nu \geq \lambda_1, \\
		-\frac{\lambda_v}{\lambda_1}, & \mbox{if } \nu = 2 \text{ and } -\lambda_\nu < \lambda_1, \\
		1 - \frac{\lambda_1}{\lambda_2}, & \mbox{if } \nu > 2 \text{ and } -\lambda_\nu \geq \lambda_1 - \lambda_2, \\
		-\frac{\lambda_\nu}{\lambda_1}, & \mbox{if } \nu > 2 \text{ and } -\lambda_\nu < \lambda_1 - \lambda_2, \\ 
	\end{array}
\right.
\]
such that $X_\eps( \tau_\eps ) = h( \Delta q_{ { \rm sgn } \psi_\eps} ) + \eps^\beta \phi_\eps$, and the random vector 
\[
\left( \psi_\eps, \phi_\eps, \tau_\eps - \frac{1}{\lambda_1} \log \eps \right)
\]
converges in distribution as $\eps \to 0$.
\end{theorem}

\begin{remark} The result is not a direct consequence of the results in this paper, since, as it will be clear at the beginning of Section~\ref{sec: ProofThmSmallNoise}, it requires a high dimensional version of the quadratic variation with drift. But we will see that to get this result the proof follows almost line by line  the proof of the main results of this paper. 
\end{remark}

\subsection{Forward-Backward Martingale Representation.} \label{sec: FB}
The proof of Theorem~\ref{thm: MainII} is based on a forward-backward martingale representation of the quadratic covariation. The focus of this section is to explain this representation as grounds to the full proof of Theorem~\ref{thm: MainII} to be given in the next section. In order to do so, we need some conventions on our notation that we state as definition:
\begin{definition}
The time reversal of a process  $X=( X(t) )_{ t\geq 0 }$  with respect to $T>0$ is defined by 
\[
\bar{X}(t)=X(T-t)-X(T), \quad t\in[0,T]. 
\]
Likewise, the backward of $X$ with respect to $T>0$ is defined by 
\[\hat X(t) = X(T-t), \quad t \in [0,T].\]
\end{definition}
%Since throughout the text, $T>0$ is an arbitrary number, the subscript on either the time reversal or backwards of a process will be omitted when the notation is clear.

%Let us give a brief review of the standard terminology in Martingale theory. Given a filtration $\mathcal{G}=( \mathcal{G}_t )_{ t\geq 0}$ satisfying the usual hypothesis, recall that a stochastic process $( X(t) )_{t\geq0}$ is a $\mathcal{G}$ martingale, if the following conditions hold:
%\begin{enumerate}
%\item $ \mathbf{E} X(t) < \infty$ for every $t \geq 0$.
%\item $X$ is $\mathcal{G}$ adapted; that is, $X(t)$ is $\mathcal{G}_t$ measurable for all $t \geq 0$.
%\item For every $0\leq s \leq t$, $ \mathbf{E} \left( X(t) | \mathcal{G}_s \right)= X(s), $ $\Pp-\rm{a.s.}$
%\end{enumerate}
%Likewise, a stochastic process $( X(t) )_{t\geq0}$ is a $\mathcal{G}$ semi martingale if it can be written as $X(t)=A(t)+M(t)$, where $( A(t) )_{t \geq 0}$ (the drift) is a right continuous with left limits $\mathcal{G}$ adapted process of locally bounded variation and $( M(t) )_{t \geq 0}$ is a $\mathcal{G}$ martingale. 

The starting point is the representation for $L_\eps=\eps^{-1} Q_\eps$ implied by~\eqref{eqn: quad-C1}%{eqn: dif_intro}
. For any $T>0$,
\begin{equation} \label{eqn: quad-int}
L_\eps (t) = - \int_0^t f( \eps W(s) ) dW(s) -\int_{T-t}^T f( \eps \hat{W} (s) ) d\hat{W}(s), \quad t\in[0,T].
\end{equation}
We will find a convenient way to rewrite this expression using an enlargement of filtration approach. Denoting the natural filtration of a process $X=(X_t)_{t\ge 0}$ by
$\Fc^X=(\Fc_t^X)_{t\ge 0}$, we note that that the integral with respect to $W$ in~\eqref{eqn: quad-int} is an
$\mathcal{F}^W$ martingale, 
while the integral with respect to $\hat W$ is the time reversal of the $\mathcal{F}^{ \hat W }$ semimartingale 
\[
N_\eps(t)= \int_0^t f( \eps \hat{W}(s) ) d \hat {W}(s).
\]
Therefore, one of the terms in~\eqref{eqn: quad-int} is a martingale, while the other one has a
nontrivial drift component. The following result reveals the structure of this time reversal.

%The idea to estimate~\eqref{eqn: quad-int} is to chose some filtrations, $\mathcal{H}$ and $\mathcal{G}$, such that $W$ and $\hat W_T$ are semi martingales with respect to $\mathcal{H}$ and $\mathcal{G}$ respectively. Additionally, we will chose this filtrations in such a way that when we use the Doob-Meyer decomposition~\cite[Section III.3]{ProtterLibro} of $W$ and $\hat W_T$ with respect 
%to $\mathcal{H}$ and $\mathcal{G}$ in~\eqref{eqn: quad-int} its drifts cancel each other out. Then we can apply martingale techniques to complete our estimate.
 
%The following is our main forward-backward martingale representation result.

\begin{theorem} \label{thm: Main_Mart}
Let $\mathcal{G}= ( \mathcal{G}_t )_{ t \in [0,T] }$ be the minimum filtration such that $W(T)$ is $\mathcal{G}_0$
measurable 
and $\mathcal{F}^{ \hat W}_t \subset \mathcal{G}_t$. Then , $\hat W$ is a $\mathcal{G}$ semimartingale with Doob--Meyer
decomposition given by
\begin{equation} \label{eqn: Doob_Meyer}
\hat{W}(t)=W(T)-\int_0^t \frac{ \hat W(s) } {T-s} ds + \beta(t),
\end{equation}
for some Brownian Motion $\beta$ with respect to $\mathcal{G}$. 

Moreover, if $\mathcal{H}=( \mathcal{H} )_{ t \in [0,T] }$ is the the minimum complete 
filtration such that $W(T)$ is $\mathcal{H}_0$ measurable and $\mathcal{F}^\beta_t \subset \mathcal{H}_t$, then $\beta$
is an $\mathcal{H}$ Brownian Motion, $\hat W$ is an $\mathcal{H}$ semimartingale with the Doob-Meyer
decomposition~\eqref{eqn: Doob_Meyer} and $\hat W$ can be written as 
\begin{equation} \label{eqn: What_sol}
\hat W (t) = W(T) ( 1 - t/T ) + ( T- t) \int_0^t \frac{ d \beta (s) }{ T-s }, \quad t \in [0,T].
\end{equation}
\end{theorem}
\begin{proof}
The result follows from Theorem~\cite[Theorem VI.3]{ProtterLibro}. 
\end{proof}
\begin{remark} \rm In particular, since $\hat W$ is $\mathcal{H}$ adapted, for every function $F:\R^2 \to \R^2$ such that 
\[
\E \int_0^T F(\hat W (s) )^2 ds< \infty,
\] 
the process $t \mapsto \int_{ T-t}^T F(\hat W(s) ) d\beta (s)$ is the time reversal of a martingale. 
\end{remark}

Using Theorem~\ref{thm: Main_Mart}, we can obtain a representation for $L_\eps=\eps^{-1} Q_\eps$. This is given in the following:
\begin{corollary} \label{cor: L}
Let $\mathcal{H}$ be as in Theorem~\eqref{thm: Main_Mart}. Then, the process $L_\eps= \eps^{-1} Q_\eps$ can be written as 
\begin{equation} \label{eqn: L_rep}
L_\eps(t)= -\int_0^t f(\eps W(s) )dW(s) - \int_{T-t}^T f( \eps \hat W(s) ) d\beta(s) + \int_0^t f( \eps W(s) ) \frac{W(s)}{s}ds,  
\end{equation}
which is the sum of a $\mathcal{F}^W$ martingale, a time reversal of an $\mathcal{H}$ martingale and a bounded variation
term.
\end{corollary}
\begin{proof}
This is an immediate consequence of Theorem~\ref{thm: Main_Mart} and~\eqref{eqn: quad-int} since $W(t)/t$ is integrable on the
interval $[0,T]$ and $f$ is bounded.
\end{proof} 
Theorem~\ref{thm: Main_Mart} is the main element we need to propose our approximation scheme, which is the main focus of Section~\ref{sec: smal_quad}. 

\section{Small Noise Analysis of Quadratic Covariation.} \label{sec: smal_quad}

In this section we study the quadratic covariation process $L_\eps = [f( \eps W),W]$. Recall the representation~\eqref{eqn: L_rep} given in Corollary~\ref{cor: L}. This will be one of the main ingredients in our proof.

Throughout this section, let $( n_{\epsilon } )_{\epsilon >0}$ be integers such that $n_\eps \nearrow \infty$ as $\eps \to 0$. Let us define $( \delta_{\epsilon } )_{\eps >0}$ by $\delta_\eps = T / n_\eps$, and observe that $\delta_\eps \searrow 0$ as  $\epsilon \to 0$.  The main result of this section is the following:
\begin{theorem} \label{thm: MainII} 
Let $q_\eps=2\sqrt{ \delta_\eps | \log \delta_\eps |}$, and let $(\gamma_\eps)_{ \eps >0 }$ satisfy $\gamma_\eps
\to 0$ and 
\begin{equation*}
 | \log \delta_\eps|\frac{ {\rm osc}_f ( \eps q_\eps ) }{ q_\eps \gamma_\eps } \to 0, \quad \eps \to 0.
\end{equation*}
Then, there are positive constants $K_1, K_2, K_3$ and $\eps_0$ such that
\begin{align*}
 \Pp \left\{  \eps^{-1} \sup_{t \leq T} |Q_\eps (t)| > \gamma_\eps  \right \} &\leq K_1 \gamma_\eps^{-1} e^ {- K_2 \gamma_\eps^2  /{\rm osc}_f ( \eps q_\eps)^2 } + K_3 \delta_\eps , \quad \eps \in (0,\eps_0 ).\\ 
\end{align*}
\end{theorem}

The idea of the proof is to start with the representation~\eqref{eqn: dif_intro} and use Theorem~\ref{thm: Main_Mart}
to prove an approximation to each integral by a sum of increments. The result will follow once we combine the
approximating sum for each integral into one. We will devote the rest of this section to developing this idea.

\subsection{Approximating Processes} \label{sec: Approx}

Let $P_{\epsilon }$ be the partition of the interval $[0,T]$ given
by points $0=s_{0}<...<s_{n_{\epsilon }}=T$, where $s_i = i \delta_{\epsilon }$, for $i=0,...,n_{\epsilon }$. Also, define the backward
partition $\hat{P}_{\epsilon }$ to be the partition of $[0,T]$
given by points $0= t_{0}<...<t_{n_{\epsilon }}=T$, where $t_{i}=T-s_{n_{\epsilon }-i}$. 

For an arbitrary process $Y$ and times $ s,t \in [0,T]$ let $\Delta _{t,s}Y=Y(t)-Y(s)$. Then, for $t \in [0,T]$ we
introduce the following notation:

\begin{align}
S_{\epsilon }(t) &=\int_{0}^{t}f( \eps W (s) )dW (s),
\label{eqn: S_def} \\
\hat{S}_{\epsilon }(t) &=\int_{T-t}^{T}f( \eps \hat W (s) )d \hat{W}(s), \label{eqn: S_hat_def} \\
J_{\epsilon }(t) &=\sum_{i=1}^{i(t)} f( \eps W(s_{i-1}))\Delta_{s_{i},s_{i-1}} W, \label{eqn: J_def} \\
\hat{J}_{\epsilon }(t) &=\sum_{i=1}^{ i(t) }  f( \eps W(s_{i}))\Delta_{s_{i},s_{i-1}} W,  \label{eqn: J_hat_def}
\end{align} 
where $i(t)$ is given by 
\[
i(t)=\min \left \{ j \in [0,n_\eps] \cap \Z : s_j \geq t  \right \}.
\]
The idea is to approximate each element $S_\eps$ and $\hat S_\eps$ with $J_\eps$ and $\hat J_\eps$ respectively, so we can approximate $L_\eps$ by $L_{\eps,P_\eps}=\hat{J}_\eps -J_\eps$. Note  that since
\begin{equation*}
f( \eps W(s_{i}))\Delta _{s_{i},s_{i-1}} W =- f( \eps \hat{W}(t_{n_\eps-i}) ) \Delta _{t_{n_{\epsilon }-i+1},t_{n_{\epsilon }-i}} \hat{W},
\end{equation*}%
after reordering the sum in (\ref{eqn: J_hat_def}), we can rewrite $\hat{J}_\eps $ as
\begin{equation}
\hat{J}_{\epsilon }(t)=-\sum_{i=n_\eps - i(t) }^{n_\eps-1} f( \eps \hat{W}(t_i) ) \Delta _{t_{i+1},t_i } \hat W,  \label{eqn: J_hat_reversed}
\end{equation}
which is an integral sum of the It\^o integral $\hat{S}_\eps$.  We will use Theorem~\ref{thm: Main_Mart} to justify the application of martingale techniques to prove that $J_\eps$ approximates $S_\eps$ and that $\hat{J}_\eps$ approximates $\hat{S}_\eps$. 
 
Once we have an approximation of $L_\eps$ by $L_{\eps, P_\eps }$, we notice that 
\begin{align}
L_{\eps, P_\eps} (t) &= \sum_{i=1}^{ i(t) }\Delta 
_{s_{i},s_{i-1}} \left( f( \eps W) \right) \Delta
_{s_{i},s_{i-1}}W . \label{eqn: Q_eps_Peps_def}
\end{align}% 
The differences in $f$ in the above expression will be used to prove that $L_{\eps, P_\eps } (t)$ converges to $0$ uniformly in probability and get the result.  

We start with some preliminary results. The proofs will be postponed until Section~\ref{sec: add_proof} in order to keep the continuity of the paper. We state the next general lemma.
\begin{lemma} \label{Lemma: Martingale_general}
Let $\left( M_\eps \right)_{\epsilon > 0}$ be a family
of martingales such that for every $\epsilon>0$, $M_\epsilon (0)=0$, the quadratic variation 
$\langle M_\eps \rangle$ is absolutely continuous with respect to Lebesgue measure, and $ \left\langle
M_{\epsilon}\right\rangle (T) \leq r_{\epsilon }$. Then, for any $\delta >0$,
\begin{equation*}
\Pp\left\{ \sup_{t\leq T}|M_{\epsilon }(t)|>\delta \right\}
< \sqrt{ \frac{8 r_\eps}{ \pi \delta^{2} }}e^ { -\delta ^{2}/ ( 2 r_\eps )} .
\end{equation*}
\end{lemma}
We give a slight generalization of Levy's modulus of continuity lemma: 
\begin{lemma} \label{lemma: Levy}
For a Brownian motion $B$, define the modulus of continuity with respect to partition $P_\eps$ by
\begin{equation} \label{eqn: Levy}
\delta_{B,\eps} = \max_{i=1,...,n_\eps} \sup_{s \in [s_{i-1},s_i]} | \Delta_{s,s_{i-1}} B |.
\end{equation}
Then, there is a constant $C>0$ independent of $\eps>0$ such that for any $\delta>0$
\[
\Pp \left \{  \delta_{B,\eps} > \delta  \right \} \leq \frac{C} { \delta\sqrt{\delta_\eps}} e^{ - \delta^2 / ( 2 \delta_\eps ) }.
\]
In particular, there is a $K_2 >0$ such that
\[
\Pp \left \{  \delta_{B,\eps} > q_\eps  \right \} \leq K_2  \delta_\eps, \quad \eps >0.
\]
\end{lemma}

With these two results at hand we are ready to estimate $L_{\eps, P_\eps }$
\begin{lemma}\label{lemma: cond_osc}

There is a positive constant $K$ such that for any $\delta>0$ and $\eps >0$,
\[
\Pp \left \{   \sup_{ t \in [0,T] } | L_{\eps,P_\eps}(t) | > \delta \right \} 
 \leq \Pp \left \{ | \log \delta_\eps|{\rm osc}_f ( \eps q_\eps ) > \frac{ q_\eps \delta}{4T}     \right \}+K \delta_\eps.
\]
\end{lemma} 
Of course, the probability in the r.h.s.\ is either $0$ or $1$, and the estimate is meaningful only if the inequality
in the curly brackets is violated.
\begin{proof}
Let us start with the simple inequality
\begin{equation} \label{eqn: quad_sum_simple}
\sup_{ t \in [0,T] } L_{\eps,P_\eps} (t) \leq \sum_{i=1}^{n_\eps } \left| \Delta_{s_i,s_{i-1} } f( \eps W) \right | \left | \Delta_{s_i,s_{i-1} } W \right|,
\end{equation}
derived from~\eqref{eqn: Q_eps_Peps_def}. We estimate each term of the sum in the r.h.s.\ of~\eqref{eqn:
quad_sum_simple}. 

From definition~\eqref{eqn: Levy} it follows that
\begin{align*}
\max_{ i=1,..,n_\eps} |\Delta _{s_i,s_{i-1}} f(\eps W)| &\leq  { \rm osc}_f ( \eps \delta _{W,\eps} ). 
\end{align*}
Using this inequality and the definition of $n_\eps$ in~\eqref{eqn: quad_sum_simple}, we see that 
\begin{align*} \notag
\sup_{ t \in [0,T] } L_{\eps,P_\eps} (t) &\leq n_\eps { \rm osc}_f ( \eps \delta _{W,\eps} ) \delta_{W, \eps}  \\
&\leq T  \delta_{W, \eps} { \rm osc}_f ( \eps \delta _{W,\eps} ) / \delta_ \eps .
\end{align*}
Hence for every $\delta > 0$ the inequalities
\begin{align} \notag 
\Pp \left \{  \sup_{ t \in [0,T] } | L_{\eps,P_\eps} | >  \delta \right \} 
& \leq \Pp \left \{  {\rm osc}_f ( \eps \delta_{W,\eps} ) \delta_{W ,\eps} >\delta_\eps \delta/T ,\  \delta_{W,\eps} \leq q_\eps \right \} + \Pp \left \{  \delta_{W,\eps} > q_\eps \right \} \\ 
& \leq \Pp \left \{  {\rm osc}_f ( \eps q_\eps ) q_\eps > \delta_\eps  \delta/ T \right \}+ \Pp \left \{ \delta_{W,\eps} > q_\eps \right \} 
\label{eqn: Q_bound}
\end{align}
hold. The second term in the r.h.s. of~\eqref{eqn: Q_bound} can be bounded using Lemma~\ref{lemma: Levy}, so we focus on
the first term. For this notice that 
\[
 {\rm osc}_f ( \eps q_\eps ) q_\eps / \delta_\eps = 4 | \log \delta_\eps | {\rm osc}_f ( \eps q_\eps )/ q_\eps,
\]
which implies that
\begin{align} \notag
\Pp \left \{  {\rm osc}_f ( \eps q_\eps ) q_\eps > \delta_\eps  \delta/ T \right \} \leq \Pp \left \{ | \log \delta_\eps|{\rm osc}_f ( \eps q_\eps ) > q_\eps \delta/(4T)     \right \} .
\end{align}
The result follows after combining this fact with~\eqref{eqn: Q_bound} and Lemma~\ref{lemma: Levy}.
\end{proof}

\subsection{Approximation of $L_\eps$ by $L_{\eps,P_\eps}$} \label{sec: approx_proof}

We have shown that $L_{\eps,P_\eps}$ converges to $0$. In order to prove the convergence of $L_\eps$ we need to prove that $L_{\eps,P_\eps}$ approximates $L_\eps$. 

In order to do so define 
\begin{equation*}
M_{\epsilon }(t):=S_{\epsilon }(t)-J_{\epsilon }(t)+f(\eps W(s_{i(t)-1}))\Delta_{s_{i(t)},t} W,  %\label{eqn: M_eps_def}
\end{equation*}%
and 
\[
\hat{M}_{\epsilon }(t):=\hat{S}_{\epsilon }(t)+ \hat{J}_{\epsilon }(t)+f(\eps \hat{W}(t_{n_\eps - i(t)
}))\Delta_{T-t,n_\eps -i(t)} \hat W.
\]
Using~\eqref{eqn: S_def},~\eqref{eqn: J_def}, and  $i(t)$, we see that the process $M_\eps$ can be written as
\[
 M_{\epsilon } (t) =\sum_{i=0}^{ n_\eps }\int_{s_{i-1} \wedge t}^{s_{i} \wedge t}\Delta _{s,s_{i-1}} f( \eps W ) dW(s).
\]
Likewise, using~\eqref{eqn: S_hat_def},~\eqref{eqn: J_hat_def},~\eqref{eqn: J_hat_reversed} and the definition of the points $t_i$, we see that 
\begin{align} \notag
\hat{M}_\eps (t) &= \sum_{ i=0 }^{ n_\eps-1 } \int_{ t_i \vee (T-t) }^{ t_{i+1} \vee (T-t) }\Delta _{s,t_i} f( \eps \hat{W} ) d\hat W(s) \\  
&= \sum_{ i=0}^{ n_\eps-1} \int_{ t_i \vee (T-t) }^{ t_{i+1} \vee (T-t) }\Delta _{s,t_i} f( \eps \hat{W} ) d\beta(s) - A_\eps (t), \label{eqn: M_hat_sum}
\end{align}
where we defined 
\begin{equation}
A_\eps (t)=\sum_{ i=0 }^{n_\eps-1} \int_{ t_i \vee (T-t) }^{ t_{i+1} \vee (T-t)}\Delta _{s,t_i} f( \eps \hat{W} ) \frac{ \hat W (s) } { T-s } ds. \label{eqn: A_def}
\end{equation}
Notice that $M_\eps$ is a $\mathcal{F}^W$ martingale and $\hat M_\eps$ is the time reversal of a $\mathcal{F}^\beta$ semimartingale. This is the main fact in the proof of the following Lemma:

\begin{lemma} \label{lemma: Foward_Approx} There are positive constants $K_1,K_2, K_3$ and $\eps_0$ such that for any $\delta >0$,
\begin{align*}
\Pp \left\{ \sup_{t\leq T}| \tilde{M}_{\epsilon }(t)|> \delta \right\} &\leq ( K_1 / \delta ) e^{- K_3 \delta^2 / {\rm osc}_f \left( \eps q_\eps \right)^{2}} + K_2 \delta_\eps , \quad \epsilon \in (0,\epsilon _0).
\end{align*}%
Here $\tilde M_\eps$ can be either $M_\eps$ or $\hat{M}_\eps$.
\end{lemma}

The following lemma will be used in the proof of Lemma~\ref{lemma: Foward_Approx}. The proof is postponed until
Section~\ref{sec: add_proof}.

\begin{lemma} \label{lemma: A}
There are positive constants $K_1, K_2, K_4$, and $\eps_0$ such that for all $\delta>0$,
\[
\Pp \left \{ \sup_{ t\in (0,T) } | A_\eps (t) | > \delta \right \} \leq ( K_1 /\delta) e^{ - K_4 \delta^2/  {\rm osc}_f ( \eps q_\eps )^2  } + K_2 \delta_\eps,  \quad \eps \in (0, \eps_0 ).
\]
\end{lemma}

\begin{proof}[Proof of Lemma~\ref{lemma: Foward_Approx} ]
Let us start with the proof for $M_\eps$.  As we said before, the process $M_\eps$ is a martingale with quadratic variation $\Gamma _{\epsilon }=\left\langle M_{\epsilon
}\right\rangle $ given by 
\begin{equation}
\Gamma _{\epsilon }(t)= \sum_{i=1}^{n_\eps} \int_{s_{i-1} \wedge t}^{s_{i}\wedge t}|\Delta _{s,s_{i-1}} f(\eps W)|^{2}ds.  \label{eqn: Gamma_def}
\end{equation}
In order to apply Lemma~\ref{Lemma: Martingale_general}, we need to find a bound on the (random) function $\Gamma_\eps$.  In this case~\eqref{eqn: Levy} implies that 
\[
\sup_{ s \in [s_{i-1},s_i]} |\Delta _{s,s_{i-1}}f( \eps W)| \leq {\rm osc}_f ( \eps \delta_{W,\eps} ),  
\]
for all $\eps>0$. Using this bound in~\eqref{eqn: Gamma_def} we see that
\begin{equation} \label{eqn: difference_quadratic_2}
\Gamma_\eps (T) \leq T {\rm osc}_f ( \eps \delta_{W,\eps} )^2 .
\end{equation}
Lemma~\ref{Lemma: Martingale_general} implies that 
\begin{align} \notag
\Pp \left\{ \sup_{t\leq T } |M_{\epsilon }(t)|> \delta \right\} &\leq \Pp \left\{ \sup_{t\leq T}|M_{\epsilon }(t)|> \delta, \Gamma_\eps (T) \leq  T  {\rm osc}_f \left( \eps q_\eps \right)^2  \right\} \\ \notag
& \quad  + \Pp \left \{  \Gamma_\eps (T) >  T {\rm osc}_f \left( \eps q_\eps \right)^2  \right \} \\ 
& \leq \sqrt{8 T \frac{  {\rm osc}_f ( \eps q_\eps)^2 } { \pi  \delta^2 } } e^{ -  \delta^2 / (2T {\rm osc}_f \left( \eps q_\eps\right)^2 ) }  + \Pp \left \{  \Gamma_\eps (T) >  T {\rm osc}_f \left( \eps q_\eps \right)^2   \right \},  \label{eqn: ineq_exp_forward}
\end{align}
for all $\eps >0$ small enough. It remains to estimate the second probability in~\eqref{eqn: ineq_exp_forward}. 
Using~\eqref{eqn: difference_quadratic_2} it easily follows that for each $\eps >0$,
\begin{align*}
\Pp \left \{  \Gamma_\eps (T) > T {\rm osc}_f \left( \eps q_\eps \right)^2   \right \} 
& \leq \Pp \left \{  {\rm osc}_f \left( \eps \delta_{W,\eps} \right)  >  {\rm osc}_f \left( \eps q_\eps \right)   
\right \} \\
& \leq \Pp \left \{  \delta_{W,\eps} >   q_\eps     \right \}.
\end{align*}
Lemma~\ref{lemma: Levy} and~\eqref{eqn: ineq_exp_forward} imply the desired estimate for $M_\eps$. 

To obtain the estimate on $\hat M_\eps$, we notice that~\eqref{eqn: M_hat_sum} and~\eqref{eqn: A_def} imply
\begin{align*}
 \hat M_\eps (T-t) +A_\eps (T-t) &=  \sum_{i=0}^{n_\eps-1} \int_{ t_i \vee t }^{ t_{i+1} \vee t}\Delta_{s,t_i} f( \eps \hat W ) d \beta (s)\\
 &= \sum_{i=0}^{n_\eps-1} \int_{ t_i }^{ t_{i+1} }\Delta_{s,t_i} f( \eps \hat W ) d \beta (s)- \sum_{i=0}^{n_\eps-1} \int_{ t_i \wedge t }^{ t_{i+1} \wedge t}\Delta_{s,t_i} f( \eps \hat W ) d \beta (s).
 \end{align*}
 Then, it follows that
 \begin{align*}
 \sup_{ t \leq T }  \left| \hat M (t)- A_\eps(t) \right| &=  \sup_{ t \leq T }  \left| \hat M (T-t)- A_\eps(T-t) \right| \\
 & \leq \left| \sum_{i=0}^{n_\eps-1} \int_{ t_i }^{ t_{i+1} }\Delta_{s,t_i} f( \eps \hat W ) d \beta (s) \right| +\sup_{ t \leq T }  \left| \sum_{i=0}^{n_\eps-1} \int_{ t_i \wedge t }^{ t_{i+1} \wedge t}\Delta_{s,t_i} f( \eps \hat W ) d \beta (s) \right| \\
 &\leq 2 \sup_{ t \leq T }  \left| \sum_{i=0}^{n_\eps-1} \int_{ t_i \wedge t }^{ t_{i+1} \wedge t}\Delta_{s,t_i} f( \eps \hat W ) d \beta (s) \right| .
 \end{align*}
Using this bound to proceed in the same way as we did for $M_\eps$, we obtain that for any $\delta >0$,
\[
\Pp \left \{ \sup_{ t \leq T }  \left| \hat M (t)- A_\eps(T) \right| > \delta \right \} \leq ( K_1/ \delta ) e^{- K_2 \delta^2 /{\rm osc}_f \left( \eps q_\eps \right)^2 } + K_2 \delta_\eps,
\]
for all $\eps >0 $ small enough. Since
\[
\Pp \left \{ \sup_{ t \leq T }  \left| \hat M (t)\right| > \delta \right \} \leq
\Pp \left \{ \sup_{ t \leq T }  \left| \hat M (t)- A_\eps(t) \right| > \delta/2 \right \}
+ \Pp \left \{ \sup_{ t \leq T }  \left| A_\eps(t) \right| > \delta/2 \right \},
\]
the result follows from Lemma~\ref{lemma: A}.
\end{proof}

A consequence of Lemma~\ref{lemma: Foward_Approx} is the approximation of the quadratic covariation $L_{\epsilon }=[f(
\eps W), W]$  by $L_{\eps,P_\eps}$, given in the following Lemma:

\begin{lemma} \label{thm: Main_Quad_Brownian} If $( \gamma_\eps)_{ \eps > 0 } $ is such that $\gamma_\eps \to 0$ and ${
\rm osc }_f ( \eps q_\eps ) q_\eps \gamma_\eps^{-1} \to 0$ as $\eps \to 0$, then there are positive constants $K_1,
K_2, K_5,$ and $\eps_0$ such that  
\begin{align*}
 \Pp\left\{ \sup_{t\leq T} |L_{\epsilon }(t)-L_{\epsilon,P_{\epsilon }} (t)|> \gamma_\eps \right\} &
\leq K_1 \gamma_\eps^{-1} e^{- K_5 \gamma_\eps^2 / {\rm osc}_f \left( \eps q_\eps \right)^2 } + K_2 \delta_\eps, \quad \epsilon \in (0,\epsilon _0).
\end{align*}
\end{lemma}

\begin{proof}
Let $(\gamma_\eps)_{ \eps >0} $ be as in the statement of the Lemma. By the definition of $M_\eps$ and $\hat{M}_\eps$, it follows that 
\begin{equation} \label{eqn: bnd_approx}
|L_{\epsilon }(t)-L_{\epsilon ,P_{\epsilon }}(t)|\leq |M_{\epsilon }|+|\hat{M%
}_{\epsilon }|+|\Delta _{s_{i(t)},s_{i(t)-1}}f( \eps W) \Delta _{s_{i(t)},s_{i(t)-1}}W |.
\end{equation}%
The result follows as a consequence of Lemmas~\ref{lemma: Levy} and~\ref{lemma: Foward_Approx}. Indeed, since 
\[
|\Delta _{s_{i(t)},s_{i(t)-1}}f( \eps W) \Delta _{s_{i(t)},s_{i(t)-1}}W | \leq {\rm osc }_f ( \eps \delta_{W,\eps} ) \delta_{W,\eps},
\]
Lemma~\ref{lemma: Levy} implies
\begin{align*}
\Pp \left \{  \right |\Delta _{s_{i(t)},s_{i(t)-1}}f( \eps W) \Delta _{s_{i(t)},s_{i(t)-1}}W | > \gamma_\eps \} &\leq \Pp \left \{ {\rm osc }_f ( \eps \delta_{W,\eps} ) \delta_{W,\eps} > \gamma_\eps, \delta_{W,\eps} \leq q_\eps \right \} \\
& \quad + \Pp \left \{  \delta_{W,\eps} > q_\eps  \right \} \\
& \leq \Pp \left \{ {\rm osc }_f ( \eps q_\eps ) q_\eps > \gamma_\eps \right \} + K_2 \delta_\eps.
\end{align*}
Hence, there is a $\eps_0>0$ such that
\[
\Pp \left \{  \right |\Delta _{s_{i(t)},s_{i(t)-1}}f( \eps W) \Delta _{s_{i(t)},s_{i(t)-1}}W | > \gamma_\eps \}  \leq K_2 \delta_\eps, \quad \eps \in (0, \eps_0).
\]
Using this bound and Lemma~\ref{lemma: Foward_Approx} in~\eqref{eqn: bnd_approx}, we obtain
\begin{align*}
\Pp\left\{ \sup_{t\leq T} |L_{\epsilon }(t)-L_{\epsilon,P_{\epsilon }} (t)|> \gamma_\eps \right\} &
\leq K_1\gamma_\eps^{-1} e^{- K_5 \gamma_\eps^2 / {\rm osc}_f \left( \eps q_\eps \right)^2 } + K_2  \delta_\eps,
\quad \eps \in (0, \eps_0 ).
\end{align*}
The proof is finished.
\end{proof}

\section{Proof of Theorems~\ref{thm: Main},~\ref{thm: Main_Lip} and~\ref{thm: MainII}}\label{sec: proof_thm}

\begin{proof}[Proof of Theorem~\ref{thm: MainII}]
The result is a consequence of Lemmas~\ref{lemma: cond_osc} and~\ref{thm: Main_Quad_Brownian}. Indeed, if $( \gamma_\eps)_{ \eps>0}$ is as in the statement of the Theorem, it is immediate to see that  
\begin{align} \notag
 \Pp \left\{  \eps^{-1} \sup_{t \leq T} |Q_\eps (t)| > \gamma_\eps  \right \} & \leq \Pp\left\{ \sup_{t\leq T}
|L_{\epsilon }(t)-L_{\epsilon,P_{\epsilon }} (t)| > \gamma_\eps/2 \right\} \\
  & \quad + \Pp \left \{   \sup_{ t \in [0,T] } | L_{\eps,P_\eps}(t) | > \gamma_\eps/2 \right \} \label{eqn: two_prob}.
\end{align}
The result will follow by applying Lemmas~\ref{lemma: cond_osc} and~\ref{thm: Main_Quad_Brownian} to the two terms in r.h.s. of~\eqref{eqn: two_prob}.

First, note that 
\[
\eta_\eps = | \log \delta_\eps | \frac{ {\rm osc}_f ( \eps q_\eps) }{ q_\eps \gamma_\eps } \to 0, \quad \eps \to 0,
\]
implies that ${\rm osc}_f ( \eps q_\eps)  q_\eps  \gamma_\eps^{-1} = 4 \eta_\eps \delta_\eps \to 0$, as $\eps \to 0$. Hence, from Lemma~\ref{thm: Main_Quad_Brownian}  we get that for some positive constants $K_1^\prime, K_2, K_5^\prime$ and $\eps_0^\prime$
\begin{align} \label{eqn: two_prob_approx}
\Pp\left\{ \sup_{t\leq T} |L_{\epsilon }(t)-L_{\epsilon,P_{\epsilon }} (t)|> \gamma_\eps/2 \right\} &
\leq K_1^{\prime} \gamma_\eps^{-1} e^{- K_5^\prime \gamma_\eps^2 / {\rm osc}_f \left( \eps q_\eps \right)^2 } + K_2 \delta_\eps,%\label{eqn: two_prob_approx}
\end{align}
for all $ \eps \in (0, \eps_0 )$ . Likewise, since $\eta_\eps \to 0$ as $\eps \to 0$, Lemma~\ref{lemma: cond_osc}
implies that for some positive constants $\eps_1$ and $K$, 
\begin{equation} \label{eqn: two_prob_conv}
\Pp \left \{   \sup_{ t \in [0,T] } | L_{\eps,P_\eps}(t) | > \gamma_\eps/2 \right \} \leq K   \delta_\eps, \quad \eps \in (0, \eps_1).
\end{equation}

The result follows by using~\eqref{eqn: two_prob_approx} and~\eqref{eqn: two_prob_conv} in~\eqref{eqn: two_prob}.
\end{proof}

\begin{proof}[Proof of Theorem~\ref{thm: Main}]
The proof is a consequence of Theorem~\ref{thm: MainII}. Indeed, let us find a family $(\delta_\eps)_{ \eps >0 }$ such that $\delta_\eps \to 0$ and
\[
 \lim_{\eps \to 0 }  |\log \delta_\eps | \frac{ {\rm osc}_f ( \eps q_\eps )}{q_\eps } =0.
\]
Let $A(\delta_\eps, \eps)=|\log \delta_\eps |  {\rm osc}_f ( \eps q_\eps )q_\eps^{-1}$. A straightforward calculation
gives 
\begin{align*}
A(\delta_\eps, \eps)&\leq  C_f \eps^\alpha \delta_\eps^{ ( \alpha -1 )/2 } | \log \delta_\eps |^{ ( \alpha + 1 )/2 }.
\end{align*}
Let $\mu \in (\gamma, \alpha)$ and take $\delta_\eps=\eps^{ 2( \alpha - \mu )/ ( 1 - \alpha ) }$. Then, $A(\eps^{ 2( \alpha - \mu )/ ( 1 - \alpha ) },\eps)\leq \hat{A}(\eps)$, where $\hat{A}(\eps)$  is given by
\[
\hat{A}(\eps)=C_{\alpha, f } \eps^{\mu} | \log \eps |^{ ( \alpha + 1 )/2 },
\]
for some constant $C_{\alpha, f}>0$ independent of $\eps>0$. So, we can use this $\delta_\eps$ in Theorem~\ref{thm: MainII} to get that
\begin{align} \notag
 \Pp \left\{  \eps^{-1} \sup_{t \leq T} |Q_\eps (t)| > \delta  \right \} &\leq K_1 \delta^{-1} \exp \left \{- C_0 \frac{ \left(\delta \eps^{-\alpha ( 1- \mu) / ( 1 - \alpha ) } \right)^2}
{ |\log \eps|^\alpha} \right \}  \\ 
& \quad + K_2  \eps^{ 2 (\alpha - \mu)/(1 - \alpha ) }, \label{eqn: ineq_Lip}
\end{align}
for all $\eps>0$ small enough and constants $K_1,K_2,C_0>0$ independent of $\eps>0$ and $\delta>0$.

Theorem~\ref{thm: MainII} actually implies that inequality~\eqref{eqn: ineq_Lip} remains true as long as 
$ \hat{A}(\eps)/\delta \to 0$, as $\eps \to 0$. So, since $\gamma \in (0,\mu)$, we can substitute $ \eps^\gamma \delta$ for $\delta$ in~\eqref{eqn: ineq_Lip} to get that 
\begin{align} \notag
 \Pp \left\{  \eps^{-( 1+\gamma) } \sup_{t \leq T} |Q_\eps (t)| > \delta  \right \} &\leq K_1 \delta^{-1} \eps^{-\gamma} \exp \left \{- C_0 \frac{ \delta^2 \eps^{2 \left(- (\alpha -\gamma) + \alpha ( \mu - \gamma) \right) / ( 1 - \alpha ) } }{  |\log \eps|^\alpha} \right \}  \\
 & \quad + K_2  \eps^{ 2(\alpha - \mu)/(1 - \alpha ) }.    \label{eqn: algebraHard}
 \end{align}
Since $\alpha \in (0,1) $ and $\mu < \alpha$, we have
 \begin{align*}
 \alpha ( \mu - \gamma) &< \alpha( \alpha -\gamma ) < \alpha - \gamma.
 \end{align*}
 Using this fact in~\eqref{eqn: algebraHard} we get that 
  \begin{align*}
 \Pp \left\{  \eps^{-( 1+\gamma) } \sup_{t \leq T} |Q_\eps (t)| > \delta  \right \} & \leq K_3 \eps^{ 2(\alpha - \mu)/(1 - \alpha ) }  ,
\end{align*}
for some $K_3>0$,  any $\delta>0$, and all $\eps >0$ small enough. The result is proved.
\end{proof}

\begin{proof}[ Proof of Theorem~\ref{thm: Main_Lip}]
 The proof follows the same steps as the proof of Theorem~\ref{thm: Main}. The first step is to to follow Theorem~\ref{thm: MainII} by finding a family $(\delta_\eps)_{ \eps >0 }$ such that $\delta_\eps \to 0$ and
\[
 \lim_{\eps \to 0 }  |\log \delta_\eps | \eps =0.
\]
Given $\gamma \in (0,1)$, we propose $\delta_\eps = e^{ - \eps^{ - (1 - \mu) } }$, for $\mu \in ( \gamma, 1)$. In this 
case, $|\log \delta_\eps | \eps=\eps^{\mu}$, so Theorem~\ref{thm: MainII} implies that
\begin{align*}
\Pp \left\{  \eps^{-1} \sup_{t \leq T} |Q_\eps (t)| > \delta  \right \} &\leq K_1 \delta^{-1} \exp \left \{- K_4 \delta^2 \eps^{ -( 1 + \mu) } e^{\eps^{- ( 1 - \mu )} } \right \}  \\ 
& \quad + K_2  e^{ - \eps^{ - (1 - \mu) } }.
\end{align*}
As in the proof of Theorem~\ref{thm: Main}, we can substitute $\delta \eps^\gamma$ instead of $\delta$ in the last
inequality. We can finish the proof by extracting the leading term in the resulting estimate.
\end{proof}

\section{Proof of Theorem~\ref{thm: SmallNoise}} \label{sec: ProofThmSmallNoise}
Using the results from~\cite{ItoC1} (see also ~\cite{Nualart}, and reference therein) we observe that $Y_\eps = h( X_\eps )$ satisfies 
\begin{align} \label{eqn: ItoC1}
dY_\eps =  \nabla h( X_\eps (t ) ) \cdot b ( X_\eps (t) ) dt + \eps \nabla h( X_\eps (t) ) dW(t) + \frac{1}{2} \mathcal{Q}_\eps(t),
\end{align}
with initial condition $Y_\eps(0) = h( X_\eps(0 ) )$, and where $\mathcal{Q} _\eps$ is an $\R^d$-valued process with j$^\text{th}$ coordinate $\mathcal{Q}^j_\eps$ given by
\begin{align}\notag 
\mathcal{Q} _\eps^j (t) &= \sum_{ k=1}^d \left[ \partial_k h^j( X_\eps ), X_\eps^j \right] \\ \label{eqn: Qjdef}
&=  \eps \sum_{ k=1}^d \left[ \partial_k h^j( X_\eps ), W^j \right], \quad j=1,...,d.
\end{align}
Differentiating with respect to $t$ the identity $h(S^tx) = e^{At}h(x)$, we get that $\nabla h( x ) b ( x ) = A h (x)$, which combined with~\eqref{eqn: ItoC1} implies
\begin{equation} \label{eqn: Yeqn}
dY_\eps(t) = AY_\eps(t) dt + \eps \left( \sigma( Y_\eps (t) ) dW(t) + \eps^{-1} \mathcal{Q}_\eps (t) \right ).
\end{equation}
From this expression, to conclude the proof it is enough to show that the term $\eps^{-1}\mathcal{Q}_\eps$ in the last display converges uniformly (in an appropriate time range) towards zero in probability. The proof of this fact extends the results of this paper, but it follows the same steps (with minor modifications) as the proof of~\cite{nhn}. 

We are now going to establish what kind of convergence we need from term $\mathcal{Q}_\eps$ in~\eqref{eqn: Yeqn} to finish the proof, and then state the result in a separate lemma. 
Using the results of~\cite{Bakhtin-SPA} that assert that $\frac{\tau_\eps}{-\log \eps}$ converges to a constant in probability, we can find for every $\upsilon>0$, there is a large enough constant $K_\upsilon>0$ such that 
\[
\Pp \left\{ \tau_\eps > -K_\upsilon \log\eps \right\} \leq \upsilon.
\]
Since $\upsilon$ is arbitrary,~\eqref{eqn: Qjdef} and~\eqref{eqn: Yeqn} imply that to finish the proof its enough to show that 
\[
\sup_{ t \in [0,-K_\upsilon \log \eps] } \left( \max_{j,k=1,...,d} \left[ \partial_j h^k ( X_\eps ), W^j \right] \right ) \to 0
\]
in probability as $\eps \to 0$. Lemma~\ref{lemma: quadh} implies this result and hence finishes the proof of this theorem.

\begin{lemma} \label{lemma: quadh}
Suppose $q:U \to \R$ is a function that satisfies the conditions of Theorem~\ref{thm: Main} with $\alpha > 1/2$. Then, for every $\Gamma>0$ and $\delta > 0$ it follows that 
\[
\lim_{ \eps \to 0} \Pp \left\{ \sup_{ t \in [0, - \Gamma \log \eps ] } [q(X_\eps), W^j](t) > \delta, \tau_\eps < -\Gamma \log \eps \right\}= 0,
\]
for every $j=1,...,d$.
\end{lemma}
\begin{proof}
The proof follow the exact same logic as the proof of Theorem~\ref{thm: Main} with slight modifications that we will point out. We keep the same notation as in Section~\ref{sec: Approx} when appropriate. For instance, $P_{\epsilon }$ is a partition of the interval $[0,-\Gamma \log \eps]$ given
by points $0=s_{0}<...<s_{n_{\epsilon }}=T_\eps = -\Gamma \log \eps$, where $s_i = i \delta_{\epsilon }$, for $i=0,...,n_{\epsilon }$. Also, define the backward
partition $\hat{P}_{\epsilon }$ to be the partition of $[0,T_\eps]$
given by points $0= t_{0}<...<t_{n_{\epsilon }}=T_\eps$, where $t_{i}=T_\eps-s_{n_{\epsilon }-i}$. 

Let us fix $j$ for the rest of the proof. Then, the idea is that the convergence towards $0$ of the process $\mathcal{q}_\eps(t) = [q(X_\eps), W^j]$ conditioned on the sigma algebra $\mathcal{A}^j_\eps$  generated by the history of $W$ up to time $T_\eps$ except for the j$^\text{th}$ component of $W$, is almost identical from the main result in Theorem~\ref{thm: Main}. We will show that this is the case, and then the proof will be finished due to the tower property of conditional expectations. 

As mentioned before, \cite{FollmerQuad},~\cite{ItoC1},~\cite{ItoCov} and~\cite{RussoCov} imply that upon fixing $\epsilon >0 $, conditioned on $\mathcal{A}^j_\eps$,
\[
\mathcal{q}_\eps (t) = - S_{\epsilon }(t) - \hat{S}_{\epsilon }(t) ,
\]
where (in analogy with the notation used in Section~\ref{sec: Approx}) we defined
\begin{align*}
S_{\epsilon }(t) =\int_{0}^{t}f( X_\eps (s) )dW^j(s), \text{ and } 
\hat{S}_{\epsilon }(t) =\int_{-\Gamma \log \eps-t}^{-\Gamma \log \eps}f( \hat X_\eps (s) )d \hat{W}^j(s).
\end{align*} 
Here the time reversal is taken with respect to time $T_\eps = -\Gamma \log \eps$.

As we did before, the proof now consists on approximating the above difference by its respective sums and then show that the approximating sequence converges to $0$. As expected to approximate the process $\mathcal{q}_\eps$ all steps will be the analogous to the ones followed in Section~\ref{sec: Approx}. In particular, $\mathcal{q}_\eps$ will be approximated by 
\begin{align}
L_{\eps, P_\eps} (t) &= \sum_{i=1}^{ i(t) }\Delta 
_{s_{i},s_{i-1}} \left( q(X_\eps) \right) \Delta
_{s_{i},s_{i-1}}W^j , \label{eqn: Q_eps_Peps_def}
\end{align}
where $i(t)$ is given by 
\[
i(t)=\min \left \{ j \in [0,n_\eps] \cap \Z : s_j \geq t  \right \}.
\]
To show that $L_{\eps,P_\eps}$ converges to $0$, we follow line by line the proof of Lemma~\ref{lemma: cond_osc}, with the only difference that $n_\eps$ is of order $-\delta_\eps^{-1} \log\eps$, and that the modulus of continuity of $X_\eps$ is now of the order $\max( \eps \delta_{\eps,W}, \delta_\eps )$. Proceeding as described, we obtain that there is a positive constant $K$ such that for any $\delta>0$ and $\eps >0$,
 \begin{equation}
 \Pp \left \{   \sup_{ t \in [0,- \Gamma \log \eps ] } | L_{\eps,P_\eps}(t) | > \delta \right \} 
 \leq \Pp \left \{ | \log \delta_\eps|{\rm osc}_q ( \max( \eps q_\eps, q_\eps^2 ) ) > \frac{ q_\eps \delta}{- 4\Gamma \log \eps}     \right \}+K \delta_\eps.
 \end{equation}
 By choosing $\delta_\eps = \eps^2$, it follows that $q_\eps$ is of the order $-\eps \log \eps$, and $\max( \eps q_\eps, q_\eps^2 )$ is of the order $-\eps^2 ( \log\eps )^2$. Hence, in this case, from the last display, to ensure that $L_{\eps,P_\eps}$ converges to $0$, we need that $\eps^{2\alpha - 1} \to 0$, as $\eps \to 0$. That is, we need $\alpha > 1/2$, as stated in the statement of the theorem.
   
 We are just left to show that the difference $L_\eps - L_{\eps, P_\eps}$ converges to $0$ under the additional conditions that $\delta_\eps$ is of order $\eps^2$. In this case, the method used in Section~\ref{sec: approx_proof} to proof Lemma~\ref{thm: Main_Quad_Brownian} follow line by line with the appropriate modifications related to the modulus of continuity of $X_\eps$, and the logarithmic grow of $\eps$. We leave the reader to fill the details.
\end{proof}

\section{Additional Proofs} \label{sec: add_proof}
\begin{proof}[Proof of Lemma~\ref{Lemma: Martingale_general}]
For each $\epsilon >0$, we use the representation of martingales as time changed Brownian Motion~\cite[Theorem
3.4.2]{Karatzas--Shreve}
to see that $M_{\epsilon }=B
(\left\langle M_{\epsilon }\right\rangle )$ in distribution in the space of
continuous functions, for some Brownian Motion $B$ (see \cite[Theorem
3.4.2]{Karatzas--Shreve}). Therefore,
\begin{equation*}
\Pp\left\{ \sup_{t\leq T }|M_{\epsilon }(t)|>\delta
\right\} \leq \Pp\left\{ \sup_{t\leq r_{\epsilon }}|B (t)|>\delta
\right\}.
\end{equation*}%
Now the symmetry of $B $, reflection principle~\cite[Section 2.6]{Karatzas--Shreve}, and Brownian
scaling (self-similarity) imply that
\begin{align*}
\Pp\left\{ \sup_{t\leq r_{\epsilon }}|B (t)|>\delta \right\}
&= \Pp \left \{ \sup_{t\leq r_{\epsilon }}\max \{ B (t) , - B (t) \} > \delta \right \} \\
&\leq 2\Pp\left\{ \sup_{t\leq r_{\epsilon }}B (t)>\delta \right\} \\
&\leq 4\Pp\left\{ B (r_{\epsilon })>\delta \right\} \\
&=4\Pp\left\{ \sqrt{r_{\epsilon }}B (1)>\delta \right\} .
\end{align*}%
The result follows by a standard Gaussian Tail estimate.
\end{proof}

\begin{proof}[Proof of Lemma~\ref{lemma: Levy}]
Fix $\delta>0$ and note that
\begin{equation} \label{eqn: Levy_sum}
\Pp \left \{ \delta_{B,\eps} > \delta \right \} \leq \sum_{i=1}^{n_\eps} \Pp \left \{ \sup_{ s \in (s_{i-1},s_i ) } | \Delta_{ s,s_{i-1} } B| > \delta  \right \}.
\end{equation}
We bound each of the probabilities in this sum. Since the process $\Delta_{ s,s_{i-1} } B$ is equal in distribution, on
the space of continuous functions, to a Brownian Motion itself up to a time shift, we can use reflection
principle~\cite[Theorem 2.9.25]{Karatzas--Shreve} and standard Gaussian bounds to get
\begin{align*}
\Pp \left \{ \sup_{ s \in (s_{i-1},s_i ) } | \Delta_{ s,s_{i-1} } B| > \delta  \right \} 
&\leq 4 \Pp \left \{ B ( \delta_\eps ) > \delta \right \} \\
&\leq \delta^{-1} \sqrt{ \frac{8 \delta_\eps} { \pi} } e^{ - \delta^2 / 2 \delta_\eps }. 
\end{align*}
Substituting this expression in~\eqref{eqn: Levy_sum} and using the fact that $n_\eps \leq 2 T/ \delta_\eps$, we see that there is a constant $C>0$ independent of $\eps>0$ such that for any $\delta>0$
\[
\Pp \left \{  \delta_{B,\eps} > \delta  \right \} \leq \frac{C} { \delta\sqrt{\delta_\eps}} e^{ - \delta^2 / ( 2 \delta_\eps ) }
\]
as expected.

To prove the second part, use $\delta= q_\eps=2\sqrt{ - \delta_\eps \log \delta_\eps }$ in the last expression to get that
\begin{align*}
\Pp \left \{  \delta_{B,\eps} >  q_\eps  \right \} &\leq  \frac{C} { 2 \delta_\eps  \sqrt{- \log \delta_\eps}}e^{ 2 \log \delta_\eps } \\
& = \frac{C \delta_\eps  } { 2  \sqrt{- \log \delta_\eps}} \\
& \leq K_2 \delta_\eps.
\end{align*}
Hence the result follows.
\end{proof}

\begin{proof}[Proof of Lemma~\ref{lemma: A}]
We start with a basic inequality
\begin{align*}
\sup_{ t \in [0,T] } |A_\eps (t)| & \leq \sum_{i=0}^{n_\eps - 1} \int_{s_i}^{s_{i+1}} | \Delta_{ s, s_i } f( \eps \hat W )| \frac{ | \hat W (s) | }{T-s}ds \\
& \leq 2 \sqrt{T} { \rm osc}_f ( \eps \delta_{W,\eps} ) \sup_{ s \leq T } \frac{ | \hat W(s) | }{ \sqrt{ T-s} } \\
& \leq 2 \sqrt{T} { \rm osc}_f ( \eps \delta_{W,\eps} ) \sup_{ s \leq T } \frac{ | W(s) | }{ \sqrt{s} }.
\end{align*}
It implies that
\begin{align*}
\Pp \left \{ \sup_{ t \in [0,T] } |A_\eps (t)| > \delta \right \} & \leq \Pp \left \{  \sup_{ s \leq T } \frac{ |W(s)| }{\sqrt{s}} > \frac{\delta} { 2{\rm osc}_f ( \eps q_\eps ) \sqrt{T} }\right \} + \Pp \left \{ \delta_{W,\eps} > q_\eps \right \} \\
& \leq \Pp \left \{  \sup_{ s \leq T } \frac{ |W(s)| }{\sqrt{s}} > \frac{\delta} { 2 {\rm osc}_f ( \eps q_\eps ) \sqrt{T} }\right \} + K_1 \delta_\eps,
\end{align*}
for some constant $K_1 > 0$ independent of $\delta>0 $ and $\eps >0$. To finish the proof, we need to study the tail probability of the random variable $A=\sup_{ s \leq T }|W(s)|/ \sqrt{s}$.

In order to study the tail decay of the random variable $A$, note that, due to the symmetry of Brownian Motion,
\[
\Pp \left\{ A > \delta \right \} \leq 2 \Pp \left\{ \sup_{ t \leq T} \frac{ W(t) }{ \sqrt{t} } > \delta \right \}.
\]
So it is sufficient to focus on the tail probabilities of the random variable $N=\sup_{t \leq T} ( W(t)/
\sqrt{t} ) $, which is the supremum of a Gaussian process.

Equip the interval $[0,T]$ with the metric $\rho$ given by 
\begin{align*}
\rho (s,t)^2 &= \mathbf{E} \left( \frac{W(s)}{\sqrt{s}} - \frac{W(t)}{\sqrt{t}} \right)^2 \\
& = 2 \left( 1- \sqrt{ \frac{s \wedge t} { s \vee t } } \right ), \quad s,t \in [0,T].
\end{align*}
We denote by $B_\theta (t) \subset [0,T]$ the $\rho$-ball of radius $\theta >0$ centered at $t  \in [0,T]$. 
Let
$H_\theta$ be the minimum number of balls of radius $\theta$ needed in order to cover $[0,T]$. According
to~\cite{Lifshits}[Section 14, Theorem 1], if
\begin{equation} \label{eqn: Dudley}
\int_0^{\sigma/2} \sqrt{ | \log H_\theta | } d\theta < \infty,
\end{equation}
with $\sigma = \sup_{ t \in [0,T] } {\rm var} ( W(t)/\sqrt{t} )=1$, then $\E N<\infty$.
Then, it is standard to see~\cite{Lifshits}[Corollary 2, Section 14] that there is a $\zeta_0> \E N$, such that for any $\zeta > \zeta_0$ 
\begin{equation} \label{eqn: concentration}
\Pp \left \{ |N-\E N| > \zeta \right \} \leq  C e^{ - \zeta^2 /2 } / \zeta,
\end{equation}
for some universal constant $C>0$. 

In our situation, if the integral in~\eqref{eqn: Dudley} is finite, this will be enough to finish the proof. Indeed, assuming~\eqref{eqn: concentration}, there is an $\eps_0 >0 $ such that
\begin{align*}
\Pp  &\left \{  \sup_{ s \leq T }  \frac{ |W(s)| }{\sqrt{s}} > \frac{\delta} { 2 {\rm osc}_f ( \eps q_\eps ) \sqrt{T} } \right \} 
 \leq 2 \Pp \left \{ N > \frac{\delta} {2 {\rm osc}_f ( \eps q_\eps ) \sqrt{T} }\right \} \\
& \hspace{1.5in}  \leq 2 \Pp \left \{ N - \E N > \frac{\delta} { 2 {\rm osc}_f ( \eps q_\eps ) \sqrt{T}} - \E N \right \} \\ 
& \hspace{1.5in}   \leq C_1\left ( \frac{ \delta }{ 2 \sqrt{T} {\rm osc}_f ( \eps q_\eps ) } - \E N \right )^{-1} \exp \left \{ - C_2 \left ( \frac{ \delta }{ 2 \sqrt{T} {\rm osc}_f ( \eps q_\eps ) } - \E N \right ) ^2 \right \} \\
& \hspace{1.5in}   \leq ( C_1 / \delta ) \exp \left \{ - C_3 \frac{ \delta ^2}{ {\rm osc}_f ( \eps q_\eps )^2} \right \},
\end{align*}
for some constants $C_1, C_2, C_3>0$ independent of $\eps$ and $\delta$, and all $\eps \in (0, \eps_0) $. Hence we just need to show that the integral~\eqref{eqn: Dudley} is finite.

We are going to give an estimate of $H_\theta$, $\theta \in (0,1/2)$. Suppose $0 \leq s < t \leq T$, then $s
\in B_\theta (t)$ if and only if 
\[
\sqrt{s} \geq \sqrt{t} ( 1 - \theta^2 / 2 ).
\] 
Therefore, if $s$ and $t$ belong to the same ball of radius $\theta \in (0,1/2)$, then
\[
| t- s| \leq T \theta^2 .
\]
Hence, $H_\theta \leq 2 / \theta^2$, and $ \sqrt{ | \log H_\theta | }$ is integrable on the interval $[0,1/2]$, which
implies our claim.
\end{proof}

\bibliographystyle{plain}
\bibliography{happydle}

\end{document}